\documentclass[a4paper,11pt]{article}

\usepackage{amsmath,amssymb,amsthm,mathtools}
\usepackage{xcolor}
\usepackage{tikz}
\usetikzlibrary{calc,positioning,backgrounds,arrows.meta}
\usepackage[left=2.6cm,right=2.6cm,top=3.0cm,bottom=3.0cm]{geometry}

\definecolor{cCenter}{RGB}{200,30,30}
\definecolor{cA}{RGB}{31,119,180}
\definecolor{cB}{RGB}{44,160,44}
\definecolor{cC}{RGB}{190,30,30}
\definecolor{cD}{RGB}{148,103,189}
\definecolor{cE}{RGB}{220,110,0}
\definecolor{cF}{RGB}{120,75,60}
\definecolor{gridC}{RGB}{175,175,175}

\theoremstyle{plain}
\newtheorem{theorem}{Theorem}
\newtheorem{lemma}[theorem]{Lemma}
\newtheorem{proposition}[theorem]{Proposition}
\newtheorem{corollary}[theorem]{Corollary}
\theoremstyle{definition}
\newtheorem{definition}[theorem]{Definition}
\theoremstyle{remark}
\newtheorem{remark}[theorem]{Remark}

\newcommand{\A}{\mathbf{A}}
\newcommand{\M}{\mathbf{M}}
\newcommand{\N}{\mathbf{N}}

\begin{document}

\title{\large\bfseries A Mean Value Property in Powers of the Natural Square\\[2pt]
\normalsize \textit{The Center Is Never Lost}}
\author{Paul Jolissaint, Kenichi Takemura}
\date{}
\maketitle
\thispagestyle{empty}

\medskip
\noindent\textbf{2020 Mathematics Subject Classification.}
15B36, 05B15, 20B25.

\medskip
\noindent\textbf{Keywords.}
Natural square matrix; magic square; matrix powers; rotational symmetry;
harmonic function; mean-value property; harmonic matrix.

\bigskip

\section{Introduction: From an Unremarkable Array}

Write the integers $1$ through $9$ into a $3\times 3$ grid, row by row,
left to right. 

\[
  \A_3 =
  \begin{pmatrix}
    1&2&3\\
	4&5&6\\
	7&8&9\\

  \end{pmatrix}.
\]

This is the \emph{natural square matrix of order $3$}.
As a mathematical object, it might seem hardly worth a second glance.
The row sums are all different; it is certainly not a magic square.
There is no obvious column symmetry either.
It is, quite simply, the most straightforward way to arrange the
integers $1$ through $9$ in a square --- a plain, unadorned table.

And yet, let us look at it from a slightly different viewpoint.
Fix the center entry~$5$, and compute the sum of the four entries that
cycle into one another under successive $90^\circ$ rotations about the center
(Figure~\ref{fig:a3orbits}).

\begin{figure}[ht]
\centering
\begin{tikzpicture}[
  cell/.style={minimum width=1.15cm, minimum height=1.15cm,
               font=\normalsize, inner sep=1pt}
]
\def\CW{1.15}\def\CH{1.15}
\foreach \rr/\cc in {0/0, 0/2, 2/2, 2/0}{
  \fill[cA!30,rounded corners=2pt]
    ({\cc*\CW-0.03},{-\rr*\CH+0.03})
    rectangle ({\cc*\CW+\CW+0.03},{-\rr*\CH-\CH-0.03});
}
\foreach \rr/\cc in {0/1, 1/2, 2/1, 1/0}{
  \fill[cB!30,rounded corners=2pt]
    ({\cc*\CW-0.03},{-\rr*\CH+0.03})
    rectangle ({\cc*\CW+\CW+0.03},{-\rr*\CH-\CH-0.03});
}
\fill[cCenter!30,rounded corners=2pt]
  ({1*\CW-0.03},{-1*\CH+0.03})
  rectangle ({1*\CW+\CW+0.03},{-1*\CH-\CH-0.03});
\foreach \x in {0,1,2,3}{
  \draw[gridC,line width=0.35pt]
    ({\x*\CW},0.03) -- ({\x*\CW},{-2*\CH-\CH-0.03});
}
\foreach \y in {0,1,2,3}{
  \draw[gridC,line width=0.35pt]
    (0,{-\y*\CH+0.03}) -- ({3*\CW},{-\y*\CH+0.03});
}
\foreach \val/\cc/\rr in {
    1/0/0,  2/1/0,  3/2/0,
    4/0/1,  5/1/1,  6/2/1,
    7/0/2,  8/1/2,  9/2/2}{
  \node[cell] at ({0.5*\CW+\cc*\CW},{-0.5*\CH-\rr*\CH}) {$\val$};
}
\node[cell,text=cCenter,font=\normalsize\bfseries]
  at ({0.5*\CW+1*\CW},{-0.5*\CH-1*\CH}) {$5$};
\node[anchor=west,align=left,font=\small] at ({3*\CW+0.4},{-0.38}) {%
  \textcolor{cA}{$\blacksquare$}\ Orbit 1 (corners):\\[3pt]
  $\quad 1+3+9+7=20=4\times5$};
\node[anchor=west,align=left,font=\small] at ({3*\CW+0.4},{-1.55}) {%
  \textcolor{cB}{$\blacksquare$}\ Orbit 2 (edge midpoints):\\[3pt]
  $\quad 2+6+8+4=20=4\times5$};
\node[anchor=west,align=left,font=\small] at ({3*\CW+0.4},{-2.65}) {%
  \textcolor{cCenter}{$\blacksquare$}\ Center entry $=5$};
\end{tikzpicture}
\caption{The two rotation orbits of $\A_3$ (shaded) and the center entry (red).
Both the corner sum and the edge-midpoint sum equal $4\times5=20$.}
\label{fig:a3orbits}
\end{figure}
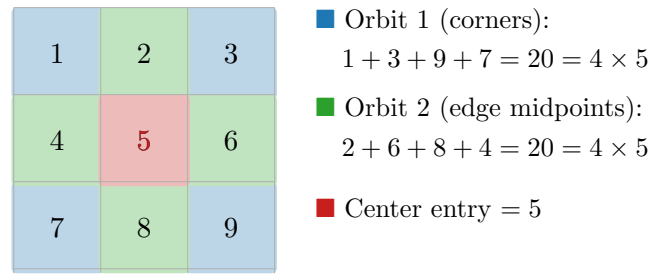

Corners: $1+3+9+7 = 20 = 4\times5$.
Edge midpoints: $2+6+8+4 = 20 = 4\times5$.
In each case, exactly four times the center entry. Moreover, summing over all its $8$ neighbours yields $8$ times its value, or, in other words, the center value is the mean of all its neighbours.

Perhaps this is not so surprising. The entries of $\A_3$ are given by
the simple linear formula $3i+j+1$, so it is not hard to believe that
entries related by rotation balance out. An orderly table, one might
reasonably say, naturally carries an orderly symmetry.

An interesting question is what happens when we apply matrix
multiplication. Let us square $\A_3$:

\[
  \A_3^2 =
  \begin{pmatrix}
    30 & 36 & 42\\
    66 & 81 & 96\\
    102 & 126 & 150
  \end{pmatrix}.
\]

The entries have become tangled together --- nothing here resembles an
orderly array anymore.
Matrix multiplication is a sweeping mixing operation, each entry
computed as a sum of products drawn from an entire row and an entire
column, like threads woven across and through one another.
It would be reasonable to expect that whatever rotational balance
$\A_3$ possessed has been scrambled beyond recognition.

Let us check. Corners: $30+42+150+102 = 324 = 4\times81$.
Edge midpoints: $36+96+126+66 = 324 = 4\times81$.
Center entry: $81$. The same conclusion holds for their mean value.

And yet --- the symmetry has not broken.

Is this a coincidence? Does it hold for $\A_3^3$ and $\A_3^4$ as well?
And what about the $n\times n$ matrix $\A$ of natural squares with $n$ odd?

The answer, as we shall see, is that it holds in every case.
Moreover, the reason has nothing to do with the particular arithmetic
of natural squares: it traces back to a single algebraic property of
their entries --- what we will call \emph{harmonicity} --- and to the
fact that matrix multiplication preserves this property faithfully
through every power.

\section{Giving the Symmetry a Name}

Let us set the stage more carefully.
For any positive odd integer $n$, write $\A_n$ for the $n\times n$ matrix
whose entries are the integers $1, 2, \dots, n^2$ arranged row by row,
left to right. Indexing rows and columns from $0$ to $n-1$, the entry
formula is simply
\[
  (\A_n)_{i,j} \;=\; ni+j+1.
\]
For $n=5$, for instance,
\[
  \A_5 =
  \begin{pmatrix}
    1 & 2 & 3 & 4 & 5\\
    6 & 7 & 8 & 9 &10\\
   11 &12 &13 &14 &15\\
   16 &17 &18 &19 &20\\
   21 &22 &23 &24 &25
  \end{pmatrix}.
\]
When $n$ is odd, there is a unique \emph{center entry} at position
$(m,m)$ where $m=(n-1)/2$: it is $5$ in $\A_3$ and $13$ in $\A_5$.

The quarter-turn clockwise rotation of the index grid about its center is the map
\[
  \rho(i,j) \;=\; \bigl(j,\; n-1-i\bigr).
\]
Working with $0\leq i\leq n-1$ as the row index (increasing downward) and $0\leq j\leq n-1$ as
the column index (increasing rightward), one can check that $\rho$
is indeed a clockwise $90^\circ$ rotation about $(m,m)$.
Since $\rho^4 = \mathrm{id}$, every index pair $(i,j)$ eventually
returns to itself after four steps.
The center $(m,m)$ is the unique fixed point of $\rho$;
every other position belongs to an orbit of length exactly~$4$:
\[
  (i,j),\quad
  (j,\,2m-i),\quad
  (2m-i,\,2m-j),\quad
  (2m-j,\,i).
\]
We call such a set a \emph{rotation orbit} (see Figure~\ref{fig:rot}).

\begin{figure}[ht]
\centering
\begin{tikzpicture}[scale=1.05,
  lbl/.style={font=\small},
  arr/.style={-{Stealth[length=6pt]},thick}
]
\def\S{1.3}
\foreach \c in {0,1,2}{
  \foreach \r in {0,1,2}{
    \draw[gridC,fill=white] (\c*\S,{-\r*\S}) rectangle ({\c*\S+\S},{-\r*\S-\S});
  }
}
\fill[cCenter!20] (1*\S,{-1*\S}) rectangle ({1*\S+\S},{-1*\S-\S});
\foreach \c/\r in {0/0, 2/0, 2/2, 0/2}{
  \fill[cA!22] (\c*\S,{-\r*\S}) rectangle ({\c*\S+\S},{-\r*\S-\S});
}
\foreach \c/\r in {1/0, 2/1, 1/2, 0/1}{
  \fill[cB!22] (\c*\S,{-\r*\S}) rectangle ({\c*\S+\S},{-\r*\S-\S});
}
\foreach \c in {0,1,2}\foreach \r in {0,1,2}{
  \draw[gridC] (\c*\S,{-\r*\S}) rectangle ({\c*\S+\S},{-\r*\S-\S});
}
\node[lbl,font=\scriptsize] at (0.5*\S,{-0.5*\S}) {$(0,0)$};
\node[lbl,font=\scriptsize] at (1.5*\S,{-0.5*\S}) {$(0,1)$};
\node[lbl,font=\scriptsize] at (2.5*\S,{-0.5*\S}) {$(0,2)$};
\node[lbl,font=\scriptsize] at (0.5*\S,{-1.5*\S}) {$(1,0)$};
\node[lbl,text=cCenter] at (1.5*\S,{-1.5*\S}) {$(1,1)$};
\node[lbl,font=\scriptsize] at (2.5*\S,{-1.5*\S}) {$(1,2)$};
\node[lbl,font=\scriptsize] at (0.5*\S,{-2.5*\S}) {$(2,0)$};
\node[lbl,font=\scriptsize] at (1.5*\S,{-2.5*\S}) {$(2,1)$};
\node[lbl,font=\scriptsize] at (2.5*\S,{-2.5*\S}) {$(2,2)$};
\def\Cx{1.5*\S}
\def\Cy{-1.5*\S}
\def\Rin{1.34*\S}
\draw[arr,cA] ({\Cx+\Rin*cos(127)},{\Cy+\Rin*sin(127)})
  arc[start angle=127,end angle=53,radius=\Rin];
\draw[arr,cA] ({\Cx+\Rin*cos(37)},{\Cy+\Rin*sin(37)})
  arc[start angle=37,end angle=-37,radius=\Rin];
\draw[arr,cA] ({\Cx+\Rin*cos(-53)},{\Cy+\Rin*sin(-53)})
  arc[start angle=-53,end angle=-127,radius=\Rin];
\draw[arr,cA] ({\Cx+\Rin*cos(-143)},{\Cy+\Rin*sin(-143)})
  arc[start angle=-143,end angle=-217,radius=\Rin];
\begin{scope}[xshift=4.4cm,yshift=-1.1cm]
  \fill[cA!35,rounded corners=2pt] (0,0) rectangle (0.38,0.38);
  \node[anchor=west,font=\small] at (0.5,0.19) {Orbit 1 (corners): length $4$};
  \fill[cB!35,rounded corners=2pt] (0,-0.65) rectangle (0.38,-0.27);
  \node[anchor=west,font=\small] at (0.5,-0.46) {Orbit 2 (edge midpoints): length $4$};
  \fill[cCenter!35,rounded corners=2pt] (0,-1.3) rectangle (0.38,-0.92);
  \node[anchor=west,font=\small] at (0.5,-1.11) {Center $(1,1)$: length $1$};
\end{scope}
\end{tikzpicture}
\caption{The index grid for $n=3$, showing the two rotation orbits and the center.
The blue arrows illustrate how $\rho$ acts on Orbit~1 (corners):
$(0,0)\to(0,2)\to(2,2)\to(2,0)\to(0,0)$.}
\label{fig:rot}
\end{figure}
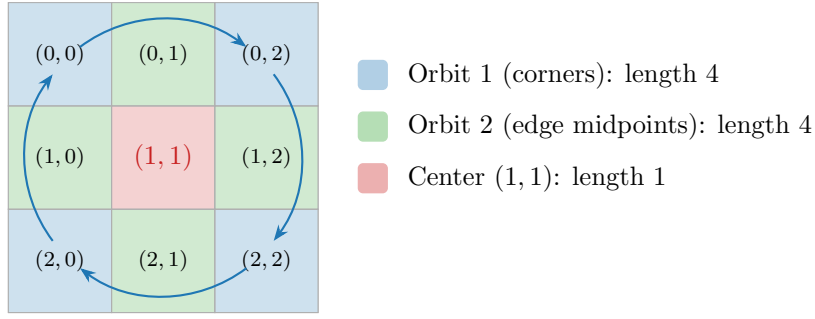

We can now state the phenomenon we noticed in $\A_3^2$ with precision.

\begin{theorem}
\label{thm:main}
Let $n$ be a positive odd integer, $\ell$ a positive integer, and $m=(n-1)/2$.
For every rotation orbit
$\{(i,j),\,(j,n{-}1{-}i),\,(n{-}1{-}i,n{-}1{-}j),\,(n{-}1{-}j,i)\}$
of $\A_n^\ell$, the sum of the four corresponding entries equals four times
the center entry:
\[
  (\A_n^\ell)_{i,j}
  +(\A_n^\ell)_{j,\,n-1-i}
  +(\A_n^\ell)_{n-1-i,\,n-1-j}
  +(\A_n^\ell)_{n-1-j,\,i}
  \;=\;4\,(\A_n^\ell)_{m,m}.
\]
\end{theorem}
As we will observe, the above result is a consequence of specific property that we call \textit{harmonicity} since it has a consequence that reminds the mean property for harmonic functions.

\section{Harmonicity: Naming the Hidden Structure}

The structure we are looking for can be described very simply.

\begin{definition}
\label{def:harmonic}
An $n\times n$ matrix $\M$ (with indices $0\le i,j\le n-1$) is called
\emph{harmonic} if there exist constants $\alpha,\beta,\gamma,\delta$ such that
\[
  M_{i,j} \;=\; \alpha\,ij+\beta\,i+\gamma\,j+\delta.
\]
In other words, $M_{i,j}$ is an affine-linear function of $i$ for each
fixed $j$, and an affine-linear function of $j$ for each fixed $i$.
For a harmonic matrix $\M$, we define its \emph{continuous extension}
\[
  \widetilde M(x,y)\;=\;\alpha xy+\beta x+\gamma y+\delta,
  \qquad (x,y)\in\mathbb{R}^2,
\]
using the same coefficients $\alpha,\beta,\gamma,\delta$: that is, we
extend $\M$ from the integer lattice $\{0,\dots,n-1\}^2$ to $\mathbb{R}^2$
by evaluating the same polynomial. By construction,
$\widetilde M(i,j)=M_{i,j}$ for $0\le i,j\le n-1$.
\end{definition}

\begin{remark}
This name has a precise mathematical justification. Extending the
indices $(i,j)$ to real variables $(x,y)\in\mathbb{R}^2$ and regarding
\[
  f(x,y)=\alpha xy+\beta x+\gamma y+\delta
\]
as a function of two real variables, we have
\[
  \frac{\partial^2 f}{\partial x^2}=0,\qquad
  \frac{\partial^2 f}{\partial y^2}=0,
\]
so that $f$ satisfies Laplace's equation
$\Delta f=\partial^2f/\partial x^2+\partial^2f/\partial y^2=0$; that is,
$f$ is a special case of a \emph{harmonic function} in the ordinary sense.

However, the four-vertex mean-value property established below at
centrally symmetric points (Lemma~\ref{lem:orbit}, Corollary~\ref{cor:square})
is \emph{not} a property of harmonic functions in general. For instance,
\[
  f(x,y)=\operatorname{Re}(z^4)
  =x^4-6x^2y^2+y^4,
  \quad z=x+iy,
\]
is also harmonic, yet the average of its values at the four vertices of
a square centered at the origin does not, in general, equal the center
value $f(0,0)=0$, as can be easily verified on $\{\pm 1,\pm i\}$ for instance.

The four-vertex mean-value property of the present paper follows not
merely from harmonicity, but from the much more special fact that
\[
  f(x,y)\in\operatorname{span}\{1,x,y,xy\},
\]
i.e., $f$ has degree at most $1$ in $x$ and in $y$ separately, and
contains no term of degree $2$ or higher in either variable. It is this
special structure that causes every term depending on the displacement
from the center to cancel when averaged over four points related by a
$90^\circ$ rotation.

More concretely, fix a center $(a,b)\in\mathbb{R}^2$ and translate
coordinates so that this center becomes the origin, writing
\[
  g(u,v)\;=f(a+u,\,b+v).
\]
Expanding, $g$ again has the same form,
\[
  g(u,v)\;=\;\delta'+\beta' u+\gamma' v+\alpha' uv,
\]
for suitable constants $\alpha',\beta',\gamma',\delta'$ depending on
$a,b,\alpha,\beta,\gamma,\delta$; in particular $\delta'=g(0,0)=f(a,b)$
is exactly the value at the center. Writing $u=r\cos\theta$,
$v=r\sin\theta$, we obtain
\[
  g(u,v)
  =
  \delta'
  +\beta' r\cos\theta
  +\gamma' r\sin\theta
  +\frac{\alpha' r^2}{2}\sin 2\theta.
\]
Summing this over the four points at angles
$\theta,\ \theta+\tfrac{\pi}{2},\ \theta+\pi,\ \theta+\tfrac{3\pi}{2}$
causes every term of angular frequency $1$ and $2$ to cancel, leaving
only the constant term $\delta'=f(a,b)$; hence the average of the four
values of $f$ at these points, each translated back by $(a,b)$, equals
the value of $f$ at the center $(a,b)$.

We emphasize that the term ``harmonic matrix'' in this paper refers
specifically to the class of matrices of the form
$M_{i,j}=\alpha ij+\beta i+\gamma j+\delta$; the corresponding function
$f(x,y)=\alpha xy+\beta x+\gamma y+\delta$ is, in addition, a genuine
harmonic function in the classical sense.
\end{remark}

The natural square $\A_n$ is immediately seen to be harmonic:
the formula $(\A_n)_{i,j}=ni+j+1$ corresponds to $\alpha=0$,
$\beta=n$, $\gamma=1$, $\delta=1$ --- a basic example of a harmonic
matrix with no $ij$ term at all.

\bigskip

Now, let $\M$ be a $n\times n$ harmonic matrix where $n$ is an odd integer, and let $m$ be as in Theorem \ref{thm:main}. For a fixed integer $1\leq k\leq m$,  let $C_k$ be the union of the sides of the square at "distance" $k$ from $(m,m)$. See Figure 3

The four sides of $C_k$ are parametrized as follows (see Figure 4):
\begin{enumerate}
\item $C_{k,1}=\{(m-k+r,m-k)\colon 0\leq r\leq 2k-1\}$ ;
\item $C_{k,2}=\{(m+k,m-k+r)\colon 0\leq r\leq 2k-1\}$ ;
\item $C_{k,3}=\{(m+k-r,m+k)\colon 0\leq r\leq 2k-1\}$ ;
\item $C_{k,4}=\{(m-k,m+k-r)\colon 0\leq r\leq 2k-1\}$.
\end{enumerate}

\begin{figure}
\begin{center}
\includegraphics*[height=5cm]{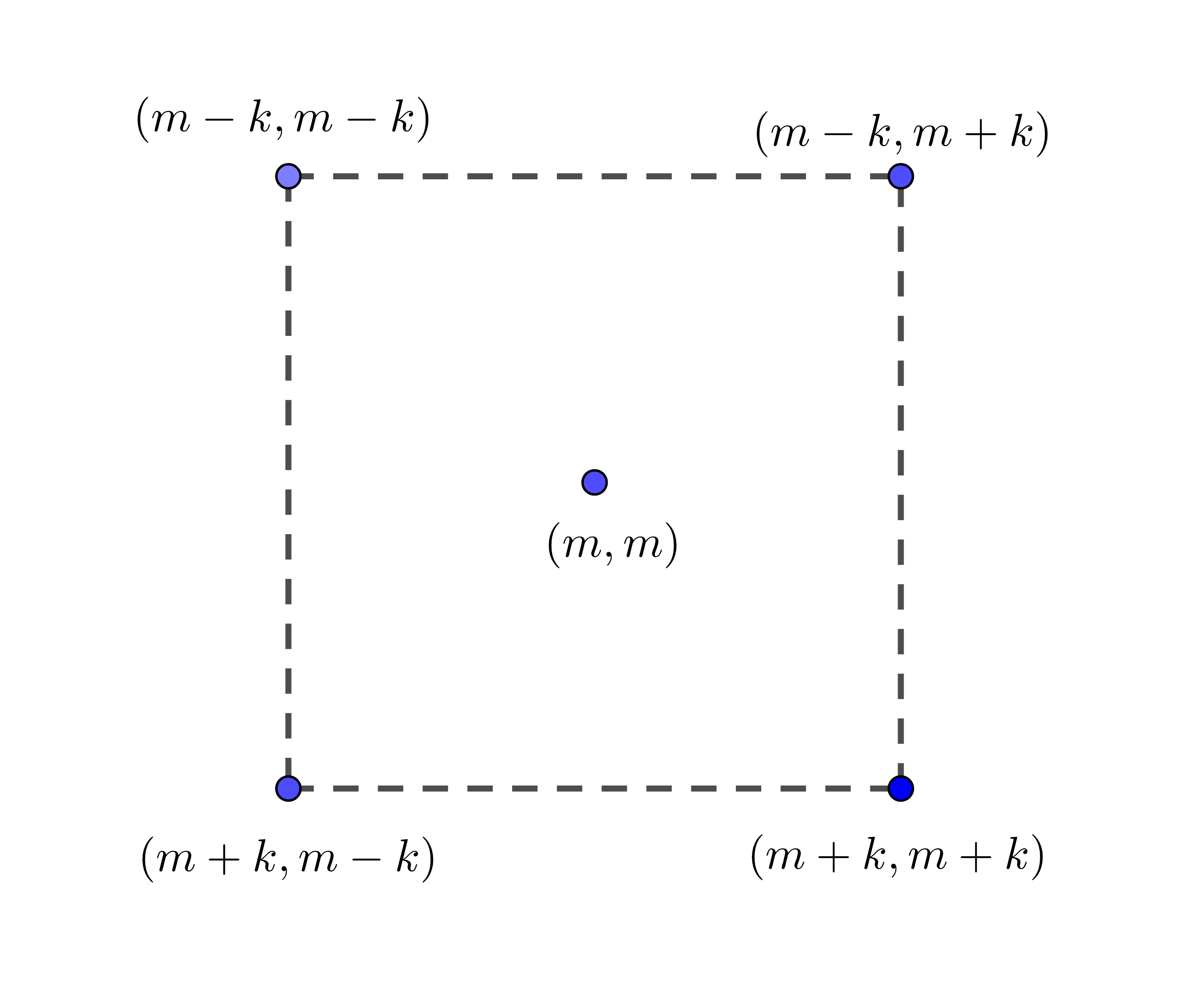}
\end{center}
\caption{The corner vertices of $C_k$ and its center.}
\end{figure}

Thus, $|C_k|=4\cdot(2k)=8k$. Then we have the following mean-value property:

\begin{theorem}\label{thm:mean}
With the above hypotheses and notation, the matrix $\M$ has the following mean value property: 

\[
\frac{1}{8k}\sum_{(i,j)\in C_k}M_{i,j}=\alpha m^2+\beta m+\gamma m+\delta=M_{m,m}.
\]
\end{theorem}

\begin{figure}
\begin{center}
\includegraphics*[height=5cm]{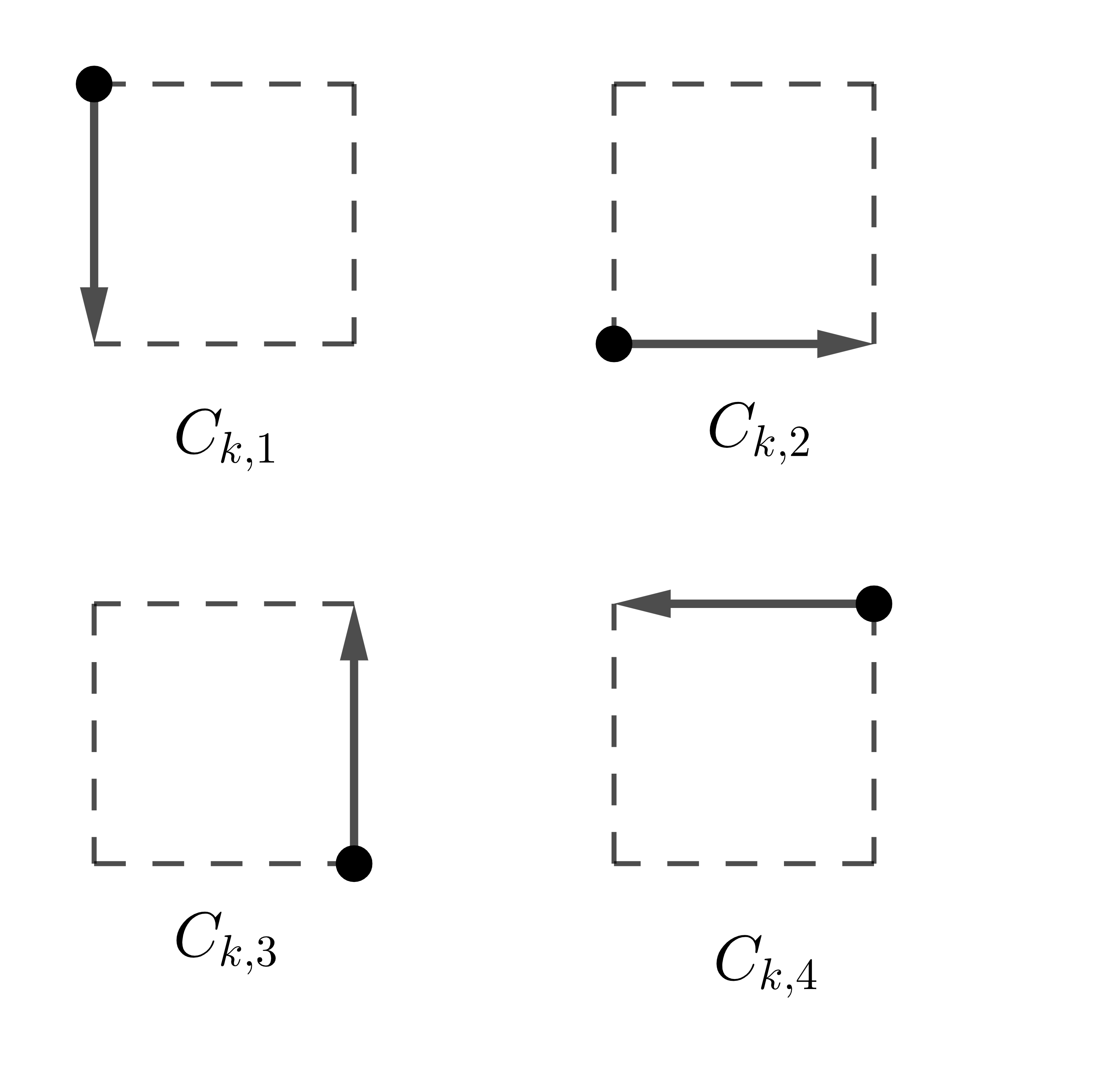}
\end{center}
\caption{The four sides of $C_k$.}
\end{figure}

\section{Proofs of Theorems~\ref{thm:main} and \ref{thm:mean}} 

We need two lemmas for the proofs of Theorems~\ref{thm:main} and~\ref{thm:mean}.

\begin{lemma}[Closure under multiplication]
\label{lem:closure}
If $\M$ and $\N$ are harmonic $n\times n$ matrices, then so is their
product $\M\N$.
\end{lemma}

\begin{proof}
Write $M_{i,j}=\alpha_1 ij+\beta_1 i+\gamma_1 j+\delta_1$ and
$N_{i,j}=\alpha_2 ij+\beta_2 i+\gamma_2 j+\delta_2$.
The $(i,j)$ entry of $\M\N$ is
\[
  (\M\N)_{i,j}
  \;=\;\sum_{\ell=0}^{n-1}
    (\alpha_1 i\ell+\beta_1 i+\gamma_1 \ell+\delta_1)
    (\alpha_2 \ell j+\beta_2 \ell+\gamma_2 j+\delta_2).
\]
Expanding and summing over $\ell$, we apply the standard identities
\[
  \sum_{\ell=0}^{n-1}1=n,\quad
  \sum_{\ell=0}^{n-1}\ell=\tfrac{n(n-1)}{2},\quad
  \sum_{\ell=0}^{n-1}\ell^2=\tfrac{n(n-1)(2n-1)}{6}.
\]
Each sum over $\ell$ becomes a constant depending only on
$\alpha_\nu,\beta_\nu,\gamma_\nu,\delta_\nu$ and $n$.
Collecting the $ij$, $i$, $j$, and constant terms gives
$(\M\N)_{i,j}=Aij+Bi+Cj+D$ for constants $A,B,C,D$.
\end{proof}

Since $\A_n$ is harmonic and the class is closed under multiplication,
it follows by induction that $\A_n^k$ is harmonic for every positive
integer~$k$.

The next lemma carries the geometric heart of the argument.

\begin{lemma}[Orbit-sum identity]
\label{lem:orbit}
Let $\M$ be a harmonic $n\times n$ matrix with $n$ odd, and let
$m=(n-1)/2$. For any index pair $(i,j)$,
\[
  M_{i,j}+M_{j,n-1-i}+M_{n-1-i,n-1-j}+M_{n-1-j,i}
  \;=\;4\,M_{m,m}.
\]
\end{lemma}

\begin{proof}
Write $M_{i,j}=\alpha ij+\beta i+\gamma j+\delta$ and $n-1=2m$.
The four orbit entries are:
\begin{align*}
  M_{i,j}            &=\alpha ij +\beta i +\gamma j +\delta,\\
  M_{j,\,2m-i}       &=\alpha j(2m-i)+\beta j+\gamma(2m-i)+\delta,\\
  M_{2m-i,\,2m-j}    &=\alpha(2m-i)(2m-j)+\beta(2m-i)+\gamma(2m-j)+\delta,\\
  M_{2m-j,\,i}       &=\alpha i(2m-j)+\beta(2m-j)+\gamma i+\delta.
\end{align*}
We sum the contribution of each coefficient separately.

\smallskip
\noindent\textit{The $\alpha$ terms:}
\begin{align*}
  &ij+j(2m-i)+(2m-i)(2m-j)+i(2m-j)\\
  &=ij+2mj-ij+4m^2-2mj-2mi+ij+2mi-ij=4m^2.
\end{align*}

\noindent\textit{The $\beta$ terms:}
$i+j+(2m-i)+(2m-j)=4m.$

\noindent\textit{The $\gamma$ terms:}
$j+(2m-i)+(2m-j)+i=4m.$

\noindent\textit{The $\delta$ terms:}
$4\delta.$

Adding everything together, the orbit sum is
\[
  4\alpha m^2+4\beta m+4\gamma m+4\delta
  \;=\;4(\alpha m^2+\beta m+\gamma m+\delta)
  \;=\;4\,M_{m,m}. \qedhere
\]
\end{proof}

\begin{proof}[Proof of Theorem~\ref{thm:main}]
By Lemma~\ref{lem:closure} and induction, $\A_n^k$ is harmonic for
every positive integer $k$. Applying Lemma~\ref{lem:orbit} to $\A_n^k$
gives the desired identity.
\end{proof}

\begin{proof}[Proof of Theorem \ref{thm:mean}]
The equality follows from the fact that $C_k$ is the union of the orbits 
\[
\{(i,j), (j,2m-i), (2m-i,2m-j), (2m-j,i): i=m-k+r, j=m-k, 0\leq r\leq 2k-1\}.
\]
See Figure 5.
\end{proof}

\bigskip

\begin{figure}
\begin{center}
\includegraphics*[height=5cm]{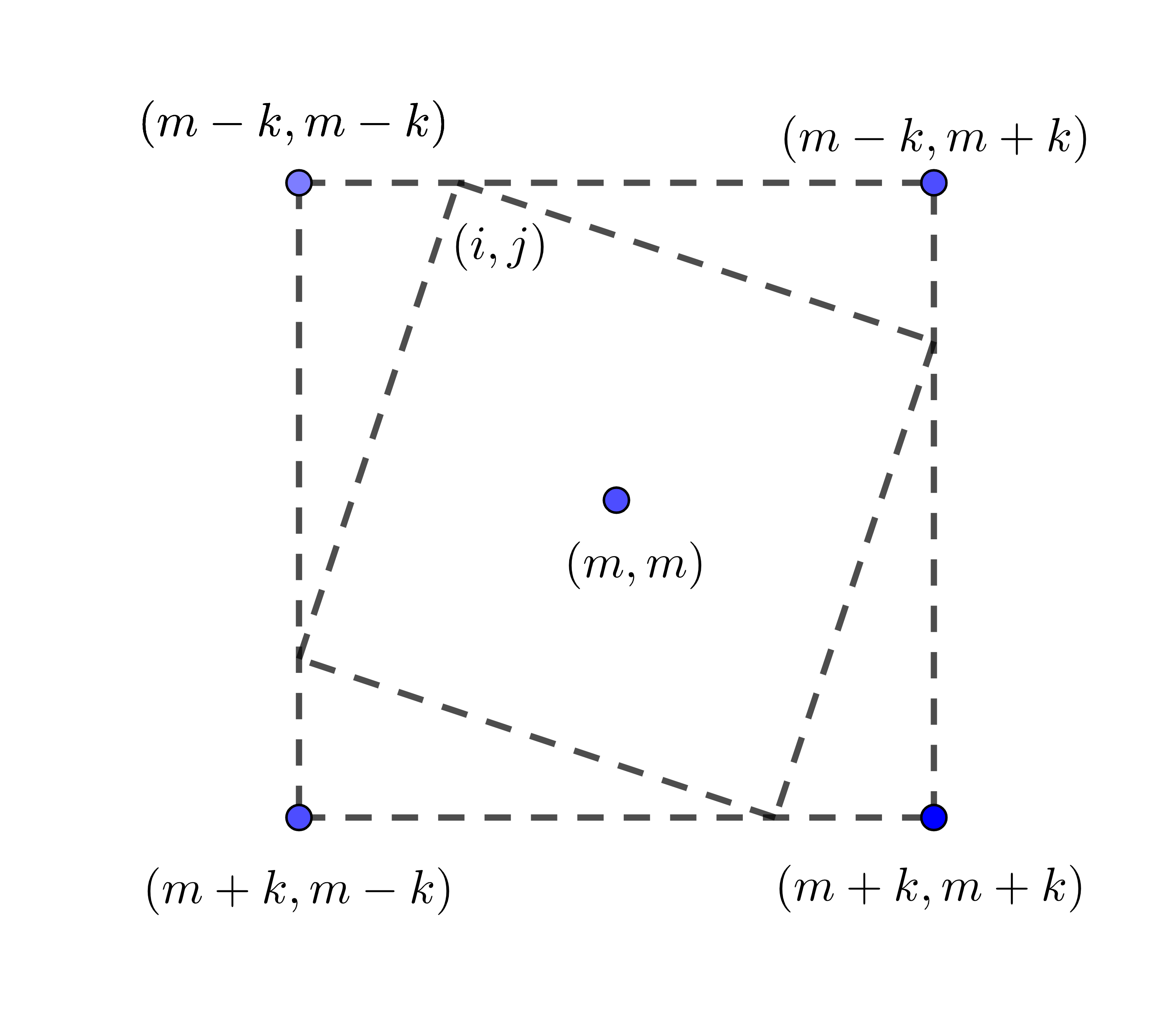}
\end{center}
\caption{$C_k$ is the union of the orbits of its upper row.}
\end{figure}

\begin{corollary}[Mean-value property over an arbitrary square]
\label{cor:square}
Let $\M$ be a harmonic $n\times n$ matrix and let $\widetilde M$ be its
continuous extension. Let $m=(n-1)/2$. For any real numbers $p,q$,
the four points related by a $90^\circ$ rotation about $(m,m)$,
\[
  (m+p,\,m+q),\quad (m-q,\,m+p),\quad (m-p,\,m-q),\quad (m+q,\,m-p),
\]
that is, the four vertices of a square of side length $\sqrt{2(p^2+q^2)}$
centered at $(m,m)$, satisfy
\[
  \frac{1}{4}\Bigl(
    \widetilde M(m{+}p,\,m{+}q)
    +\widetilde M(m{-}q,\,m{+}p)
    +\widetilde M(m{-}p,\,m{-}q)
    +\widetilde M(m{+}q,\,m{-}p)
  \Bigr)
  \;=\;\widetilde M(m,m).
\]
(When $n$ is odd, $m$ is an integer and $\widetilde M(m,m)=M_{m,m}$ is
the actual center entry of $\M$. The even case is treated in
Remark~\ref{rem:even}.)
\end{corollary}

\begin{proof}
In the proof of Lemma~\ref{lem:orbit}, replace $M_{i,j}$ by
$\widetilde M(x,y)=\alpha xy+\beta x+\gamma y+\delta$, and the integer
indices $i,j$ by the real numbers $x=m+p$, $y=m+q$. The computation
consists only of polynomial arithmetic in $\alpha,\beta,\gamma,\delta$
and never uses that $i,j$ are integers, so the same calculation applies
verbatim.
\end{proof}

Thus the ``rotation orbits'' treated by Lemma~\ref{lem:orbit} are
merely those squares centered at $(m,m)$, among the infinitely many
such squares, for which $p,q$ are integers small enough to stay inside
the grid. What is essential is not the grid-specific structure of a
rotation orbit, but the algebraic property we have called
\emph{harmonicity}: that each entry is a polynomial
$\alpha ij+\beta i+\gamma j+\delta$ of degree at most $1$ in $i$ and in
$j$ separately. The same mean-value property holds for a square of any
size and orientation about the center.

\begin{remark}
The identity in Lemma~\ref{lem:orbit} holds for \emph{every} index
pair $(i,j)$, not only those forming a rotation orbit of length~$4$.
Substituting $(i,j)=(m,m)$ gives the tautology $4M_{m,m}=4M_{m,m}$,
which simply says that the center fits consistently into the picture
as a degenerate orbit of length~$1$.
\end{remark}

\section{Beyond the Natural Square}

Looking back at the argument, one realizes that the natural square
played a rather modest role: it was simply one example of a harmonic
matrix. What is essential is not the natural square as a starting
point, but the algebraic property of belonging to the class of
harmonic matrices defined above, and this observation opens the door
to a more general statement.

Lemma~\ref{lem:orbit} holds for \emph{any} harmonic matrix, using
nothing more than the defining property that each entry has the form
$\alpha ij+\beta i+\gamma j+\delta$. Corollary~\ref{cor:square}, its
generalization via the continuous extension $\widetilde M$, likewise
holds for any harmonic matrix. Lemma~\ref{lem:closure}, meanwhile,
guarantees that the class of harmonic matrices is closed under matrix
multiplication. The orbit-sum identity therefore propagates beyond
the natural square to every power of every matrix in the class of
harmonic matrices. Indeed, the argument is not specific to the
natural square at all:

\begin{proposition}
Let $n$ be a positive odd integer and $\M$ a harmonic $n\times n$
matrix. Then $\M^k$ is harmonic for every positive integer $k$, and
every rotation-orbit sum of $\M^k$ equals four times its center entry.
\end{proposition}

Theorem~\ref{thm:main} is the special case $\M=\A_n$.

\begin{remark}
\label{rem:even}
When $n$ is even, $m=(n-1)/2$ is not an integer, and there is no
center entry on the grid. However, the proof of Lemma~\ref{lem:orbit}
never uses that $m$ is an integer; it relies only on the algebraic
relation $n-1=2m$. Hence, using the continuous extension
$\widetilde M(x,y)=\alpha xy+\beta x+\gamma y+\delta$, the very same
identity
\[
  \widetilde M(i,j)+\widetilde M(j,\,n{-}1{-}i)
  +\widetilde M(n{-}1{-}i,\,n{-}1{-}j)+\widetilde M(n{-}1{-}j,\,i)
  \;=\;4\,\widetilde M(m,m)
\]
holds when $n$ is even as well. This is not merely an analogous
statement; it is a direct application of the same identity. The
right-hand side $\widetilde M(m,m)$ is not an actual matrix entry,
since $m$ is not a lattice point, but it is meaningful as the
\emph{virtual center value} determined by the continuous extension of
the harmonic matrix.
\end{remark}


\begin{thebibliography}{9}

\bibitem{andrews1917}
W.~S.~Andrews,
\emph{Magic Squares and Cubes},
2nd ed., Open Court Publishing, Chicago, 1917;
reprinted by Dover Publications, New York, 1960.

\bibitem{hilllettingtonschmidt2018}
S.~L.~Hill, M.~C.~Lettington, and K.~M.~Schmidt,
Block representations and spectral properties of constant sum matrices,
\emph{Electron. J. Linear Algebra} \textbf{34} (2018), 170--190.

\bibitem{thompson1994}
A.~C.~Thompson,
Odd magic powers,
\emph{Amer. Math. Monthly} \textbf{101} (1994), no.~4,
339--342.

\bibitem{ward1980}
J.~E.~Ward~III,
Vector spaces of magic squares,
\emph{Math. Mag.} \textbf{53} (1980), no.~2, 108--111.

\end{thebibliography}

\renewcommand{\refname}{References}

\bigskip
\section*{Acknowledgments}
We are grateful to the editors of \emph{Elemente der Mathematik} for bringing the authors together and for providing the opportunity that led to this collaboration.

\medskip
\noindent\textbf{Competing interests.}
The authors declare no competing interests.

\medskip
\noindent\textbf{Funding.}
This research received no specific grant from any funding agency
in the public, commercial, or not-for-profit sectors.

\medskip
\noindent\textbf{Data availability.}
No external data sets were used in this research. All numerical results
can be reproduced from the formulas in the text.

\medskip\noindent
\textbf{Addresses}
P.J. Universit\'e de Neuch\^atel, Institut de Math\'ematiques, rue E.-Argand 11, CH-2000 Neuch\^atel, \texttt{paul.jolissaint@ik.me, pajolissaint@gmail.com}
ORCID: 0009-0009-1786-0983\\
K.T. Independent Researcher, Tokyo, Japan, \texttt{mahora1123@gmail.com} ORCID: 0009-0004-3111-8826.

\end{document}